\documentclass[a4paper,12pt]{amsart}

\usepackage{amsthm,amsfonts}
\usepackage{amssymb,graphicx,color}
\usepackage[all]{xy}

\newtheorem{theorem}{Theorem}[section]
\newtheorem{lemma}[theorem]{Lemma}
\newtheorem{corollary}[theorem]{Corollary}

\newtheorem{definition}[theorem]{Definition}

\newtheorem{remark}[theorem]{Remark}

\newcommand{\R}{\mathbb R}

\newcommand{\N}{\mathbb N}

\begin{document}

\title[]{Finiteness theorem for multi-$\mathcal{K}$-bi-Lipschitz equivalence of map germs}

\author[L. Birbrair]{Lev Birbrair*}\thanks{*Research supported under CNPq 302655/2014-0 grant and by Capes-Cofecub}
\address{Departamento de Matem\'atica, Universidade Federal do Cear\'a
(UFC), Campus do Picici, Bloco 914, Cep. 60455-760. Fortaleza-Ce,
Brasil} \email{birb@ufc.br}

\author[J.C.F. Costa]{Jo\~{a}o Carlos Ferreira Costa**}\thanks{**Research supported by FAPESP and CAPES}
\address{UNESP - C\^ampus de S\~ao Jos\'e do Rio Preto,
Rua Crist\'ov\~ao Colombo, 2265 - Jardim Nazareth 15054-000 S\~ao
Jos\'e do Rio Preto-SP. } \email{jcosta@ibilce.unesp.br }

\author[E.S. Sena Filho]{Edvalter Da Silva Sena Filho***}\thanks{***Research supported by Capes}
\address{Departamento de Matem\'atica, Universidade Estadual Vale do Acara\'u.
(UVA), Avenida Doutor Guarani - at\'e 609/610, Cep. 62042-030. Sobral-Ce,
Brasil} \email{edvalter.filho@hotmail.com}

\author[R. Mendes]{Rodrigo Mendes***}
\address{Departamento de Matem\'atica, Universidade de Integra\c{c}\~ao Internacional da Lusofonia Afro-Brasileira (unilab)
, Campus dos Palmares, Cep. 62785-000. Acarape-Ce,
Brasil} \email{rodrigomendes@unilab.edu.br}



\begin{abstract}

Let $P^{k}(n,p) $ be the set of all real polynomial map germs $f = (
f_1 , ..., f_p ) : (\mathbb{R}^{n},0) \rightarrow
(\mathbb{R}^{p},0)$ with degree of $f_1 , ...,f_p$ less than or
equal to $ k \in \mathbb{N}$. The main result of this paper shows
that the set of equivalence classes of $ P^{k}(n,p)$, with respect
to multi-$\mathcal{K}$-bi-Lipschitz equivalence, is finite.
\end{abstract}

\maketitle

\section{Introduction}

Given an equivalence relation to classify map germs in the context
of Singularity theory, one initial problem is the following:

\textbf{Problem.} To decide if the classification under
investigation has or not countable number of equivalence classes, 
becoming finite under some restrictions.

If the answer is affirmative, we say that the classification is tame
or has the finiteness property. The finiteness property is an
initial step in any possible attempt of the understanding the classification of 
singularities.

Let $P^{k}(n,p) $ be the set of all real polynomial map germs $f = (
f_1 , ..., f_p ) : (\mathbb{R}^{n},0) \rightarrow
(\mathbb{R}^{p},0)$ with degree of $ f_1 , ..., f_p$ less than or
equal to $ k \in \mathbb{N}$. For some equivalence relations the finitness property does not hold. Henry and Parusinski \cite{HP}
showed existence of moduli for $\mathcal{R}$-bi-Lipschitz
equivalence of analytic functions. By other hand, in \cite{LJAR} the
authors showed that with respect to $\mathcal{K}$-bi-Lipschitz
equivalence, the finiteness property holds for the set of all real
polynomial map germs with bounded degree. Later on, Ruas and Valette
\cite{RV} showed a more general result about the finiteness of
Lipschitz types with respect to $\mathcal{K}$-bi-Lipschitz
equivalence of map germs.

In this paper we introduce the notion of
multi-$\mathcal{K}$-bi-Lipschitz equivalence to investigate the
finiteness property in the set $P^{k}(n,p)$.  This equivalence is closed 
related with the notion of contact equivalence for
$q$-tuple of map germs introduced by Sitta in \cite{sitta}. The
approach of \cite{sitta} was motivated by Dufour's work
\cite{dufour}.

A pair (or couple) of map-germs can be defined as follows:
$$(f,g):(\R^n,0) \rightarrow
(\R^p \times \R^q,0).$$ It can be seen also as a divergent
diagram
$$
(\R^q,0) \stackrel{g}{\longleftarrow} (\R^n,0)
\stackrel{f}{\longrightarrow} (\R^p,0).
$$

Divergent diagrams appear in several geometrical problems and it
have many applications.

Recently, in \cite{LJE} the authors showed the finiteness property
for the called bi-$C^0$-$\mathcal{K}$-equivalence of pairs of map
germs. This equivalence relation is the topological version of
topological contact equivalence adopted to a pair of map germs.  An overview of the theory involving classical equivalence
relations of pairs of map germs can be found in \cite{CPS}.

The main result of this paper shows that the set of equivalence
classes of $ P^{k}(n,p)$, with respect to
multi-$\mathcal{K}$-bi-Lipschitz equivalence, is finite. As a
consequence, we obtain the finiteness property with respect to
$\mathcal{K}$-bi-Lipschitz equivalence.

\section{Preliminaries and notations}

\begin{definition}\label{defKbil} Two continuous map germs
$f,g:(\R^n,0) \to (\R^p,0)$ are said to be
\emph{$\mathcal{K}$-bi-Lipschitz equivalent} if there exist
 germs of bi-Lipschitz homeomorphisms \linebreak $H: (\R^n \times \R^p, 0) \rightarrow (\R^n
\times \R^p, 0)$ and $h:(\R^n, 0) \rightarrow (\R^n, 0)$ such that
$H(\R^n \times \{0 \}^p) = \R^n \times \{0 \}^p$ and the following
diagram is commutative:
\[
\begin{array}{lllll}
(\R^n,0) & \stackrel{(id_n,f)}{\longrightarrow}
 & (\R^n \times
\R^p,0) & \stackrel{\pi_n}{\longrightarrow} & (\R^n,0) \\
\,\,\,h \, \downarrow &  & \,\,\,\, H \, \downarrow  &  & \,\,\, h \, \downarrow \\
(\R^n,0)& \stackrel{(id_n,g)}{\longrightarrow} & (\R^n \times
\R^p,0) & \stackrel{\pi_n}{\longrightarrow}& (\R^n,0) \\
\end{array}
\]
\vspace{0,1cm}

\noindent where  $id_n:(\R^n,0)  \rightarrow (\R^n,0)$ is the
identity map germ of $\R^n$, $\pi_n: (\R^n \times \R^p,0)
\rightarrow (\R^n,0)$ is the canonical projection germ and $\{0\}^p
= (0, \dots,0) \in \R^p$.

When $h = id_{n}$, $f$ and $g$ are said to be
\emph{$\mathcal{C}$-bi-Lipschitz equivalent}.
\end{definition}

In other words, two map germs $f$ and $g$ are
$\mathcal{K}$-bi-Lipschitz equivalent if there exists a germ of
bi-Lipschitz map $H:(\R^n \times \R^p, 0) \rightarrow (\R^n \times
\R^p, 0)$ such that $H(x,y)$ can be written in the form
$H(x,y)=(h(x), \theta(x,y))$, $x \in \R^n, y \in \R^p$, where $h$ is
also a bi-Lipschitz map germ, such that $\theta(x,0)=0$ and $H$ maps
the germ of the graph$(f)$ onto the graph$(g)$. Recall that
graph$(f)$ is the set defined as follows:
$$
{\rm graph}(f) = \{(x,y) \in \R^n \times \R \, | \, y=f(x)\}.
$$

\begin{definition} Two functions $ f, g :\mathbb{R}^{n} \rightarrow \mathbb{R}$ are called \emph{of the same contact at a point}
$ x_0 \in \mathbb{R}^n$ if there exist a neighborhood $ U_{x_0}$ of
$x_0$ in $\mathbb{R}^{n}$ and two positive numbers $c_1$ and $c_2$
such that, for all $ x \in U_{x_0}$, we have

$$ c_1 \ f(x) \leq g(x) \leq c_2 \ f(x).$$

We use the notation: $ f \approx g$.
\end{definition}

The next Lemma is an adaptation of Theorem 2.4 given in \cite{LJAR}:

\begin{lemma}\label{lema1} Let $ f , g : (\mathbb{R}^{n},0) \rightarrow (\mathbb{R},0) $ be two germs of Lipschitz functions.
Suppose that there exists a germ of bi-Lipschitz homeomorphism $ h:
(\mathbb{R}^{n},0) \rightarrow (\mathbb{R}^{n},0) $ such that one of
the following two conditions is true:\vspace{0,15cm}

i) $ f \approx g \circ h $ or \vspace{0,15cm}

ii) $ f \approx - g \circ h $. \vspace{0,15cm}

Then, $f$ and $g$ are $\mathcal{K}$-bi-Lipschitz equivalent.
\end{lemma}

\begin{proof} Suppose $f \approx g \circ h$ (the case $f \approx - g \circ h$ is analogous).
Let $ H : (\mathbb{R}^{n} \times \R,0) \rightarrow (\mathbb{R}^{n}
\times \R,0)$ given by \vspace{0,5cm}

$H(x,y) =$
$\left \{
\begin{array}{ll}
( h(x),  0 ) & {\rm if} \,\,\, y = 0  \\
( h(x),  \frac{ {g} \circ h(x) y }{ f (x) } ) & {\rm if} \,\,\,  0 <
\vert y
\vert \leq \vert f(x) \vert \\
( h(x), y - f(x) + g \circ h (x))  & {\rm if} \,\,\,  0 < \vert f(x) \vert \leq \vert y \vert \,\,\,\,\, \mbox{and} \,\,\,\,\, sign(y)= sign(f(x))\\
( h(x), y + f(x) - g \circ h (x))  & {\rm if} \,\,\,  0 < \vert f(x)
\vert
\leq \vert y \vert \,\,\,\,\, \mbox{and} \,\,\,\,\, sign(y)= -sign(f(x))\\
( h(x),  y )  & {\rm if} \,\,\,  f(x) = 0. \\
 \end{array}
 \right.
$
\vspace{0,5cm}

The map $H(x,y) = (h(x),\theta(x,y) )$ defined above is
bi-Lipschitz. In fact, $H$ is injective because, for any fixed
$x^{*}$, we can show that $\theta( x^{*}, y )$ is a continuous and
monotone function. Moreover, $H$ is Lipschitz if $ 0 < \vert f(x)
\vert \leq \vert y \vert.$ Let us show that $H$ is Lipschitz if $0 <
\vert y \vert \leq \vert f(x) \vert.$ Hence it is sufficient to show
that all the partial derivatives of $H$ exist and are bounded in
their domain up to a set of measure zero. Since $h$ is a
bi-Lipschitz homeomorphism, follows that all its partial derivatives
$\displaystyle \frac{\partial{h} }{
\partial x_i } $ exist and are bounded in almost every $x$ near $0
\in \mathbb{R}^n$, $i=1, \dots,n$. Thus, it is necessary to check
only the partial derivatives of $\theta$. Then,

$$
\left.
\begin{array}{lll}
\displaystyle \frac{ \partial{\theta} }{ \partial x_i } & = & \displaystyle \frac{ \left({\sum}_{ _{j = 1}}^{n} \frac{ \partial g}{ \partial x_j }(h(x))  \frac{ \partial h_j }{ \partial x_i } (x)f(x) - \frac{ \partial f}{ \partial x_i }(x) g \circ h(x)\right)y}{  {(f(x))}^{2} }\\
 & & \\


 & = &  \displaystyle {\sum}_{ _{j = 1}}^{n} \frac{ \partial g}{ \partial x_j }(h(x)) \frac{ \partial h_j }{ \partial x_i } (x) \frac{y}{f(x)} -  \frac{ \partial f}{ \partial x_i }(x) \frac{ g \circ h(x)}{f(x)} \frac{y}{f(x)}. \\
 & & \\
 \end{array}
 \right.
$$

Observe that,

\medskip

i) $ \displaystyle \frac{ \partial h_j }{ \partial x_i } (x)$ is
bounded up to a set of measure zero, for all $ \ i, j = 1,..., n $;

ii) since $ 0 < \vert y \vert \leq \vert f(x) \vert $ then
$\displaystyle \frac{y}{f(x)}$ is bounded;

iii) since $f \approx g \circ h$, the expression $\frac{ g \circ
h(x)}{f(x)}$ is bounded;

iv) $\displaystyle \frac{\partial g}{ \partial x_j }, \frac{
\partial f}{\partial x_i }$ are bounded for all $ i,j = 1,..., n, $
because $f$ and $g$ are Lipschitz.

\medskip

From (i)-(iv)  we can conclude the $\displaystyle
\frac{\partial{\theta} }{
\partial x_i }$ exist and it is bounded up to a set of measure zero, for all $i=1, \dots,
n$. Observe that the partial derivative of $\theta$ with respect
to $y \in \R$ is also bounded.

Since the map $H$ is Lipschitz outside of the set of the measure zero,
the following expression holds:
$$ \parallel H(x,y) - H(u,v) \parallel \leq k \parallel (x,y) -
(u,v) \parallel \,\,\,\,\,\,\, \forall \ (x,y) , (u,v) \in  V
\setminus U, $$

\noindent where $V$ is an open neighborhood of the origin in
$(\mathbb{R}^{n} \times \mathbb{R},0)$, $U$ is a set of measure
zero and $k$ is a real constant positive. We need to show that
the last inequality remains valid for all $(x,y) , (u,v) \in V$. In
fact, take $ ( x_0 , y_0 ) , ( u_0 , v_0 ) \in V \cap U.$ Since $U$
is a set of measure zero, there exist sequences $( x_n , y_n ), (
u_m , v_m ) \in V \setminus U$ such that $( x_n , y_n ) \rightarrow
( x_0 , y_0 )$ e $( u_m , v_m ) \rightarrow ( u_0 , v_0 )$.

Moreover,
$$\parallel H( x_n , y_n ) - H( u_m , v_m ) \parallel \leq k
\parallel ( x_n , y_n ) - ( u_m , v_m ) \parallel, \,\,\,\,\,\,\, \forall \ n , m
\in \mathbb{N}.$$

Since $H$ is a continuous map, taking  $n \rightarrow \infty $ and
then $ m \rightarrow \infty$ follows that
$$ \parallel H( x_0 , y_0 ) - H( u_0 , v_0 ) \parallel \leq k
\parallel ( x_0 , y_0 ) - ( u_0 , v_0 )
\parallel.$$

Hence, $H$ is Lipschitz in all $V$.

Since $H^{-1}$ can be constructed in the same form as $H$, we
conclude that  $H^{-1}$ is also Lipschitz. Then, $H$ is
bi-Lipschitz. Moreover, by construction of $H$, follows that:

i) $H(x, f(x) ) = (h(x), g \circ h(x))$ and \vspace{0,2cm}

ii) $H( \mathbb{R}^n \times \{ 0 \}) = \mathbb{R}^n \times \{ 0 \}$.
\vspace{0,2cm}

Hence $f$ and $g$ are $\mathcal{K}$-bi-Lipschitz equivalent.
\end{proof}

\begin{definition} Let $q = (q_1, q_2, ... ,q_p ) \in \mathbb{N}^p $. A \emph{$(p,q)$-multi pair of map germs} is a family of
$p$ maps $\{F_1, \dots, F_p\}$, where $F_i : ( \mathbb{R}^n ,0)
\rightarrow ( \mathbb{R}^{q_i} ,0)$, $i=1, \dots, p$ is a  map germ.
A $(p,q)$-multi pair of map germs can be considered as a map germ
$F= ( F_1 , ... , F_p ):(\R^n,0) \to (\R^{q_1} \times \R^{q_2}
\times \dots \times \R^{q_p},0)$, where $q$ is called  \emph{multi
index}  of a $(p,q)$-multi pair of map germs.
\end{definition}

\begin{remark}\label{remark} A map germ $f:(\R^n,0) \to (\R^p,0)$ can be considered as a $(p,q)$-multi pair of
function germs, i.e., take $q=(1, \dots,1)$ $p$-times.
\end{remark}

\begin{definition}\label{defcontactfamily} A family of $p+1$  germs of bi-Lipschitz homeomorphisms of type  $( h, H_1 , ... , H_p ),$
where $ h : ( \mathbb{R}^n ,0) \rightarrow ( \mathbb{R}^n ,0)$ and $
H_i : ( \mathbb{R}^n \times \mathbb{R}^{q_i} , 0) \rightarrow (
\mathbb{R}^n \times \mathbb{R}^{q_i} , 0)$ is called a \emph{contact
family} if $H_i (x, y_i ) = ( h(x), \tilde{H_i} (x, y_i ))$ for all
$i = 1, ..., p$, $x \in \R^n$, $y_i \in \mathbb{R}^{q_i}$.
\end{definition}

\begin{definition}\label{defH} Let $ q = ( q_1 , ... , q_p )$ be a multi index. A \emph{$(p,q)$-multi pair of bi-Lipschitz homeomorphisms} is a germ of a bi-Lipschitz homeomorphism,
generated by a contact family of $p+1$ bi-Lipschitz homeomorphisms
in the form
$$\mathcal{H} : ( \mathbb{R}^{n} \times \mathbb{R}^{q_1}
\times ... \times \mathbb{R}^{q_p}  ,0) \rightarrow ( \mathbb{R}^{n}
\times \mathbb{R}^{q_1} \times ... \times \mathbb{R}^{q_p}  ,0),$$
\noindent given by
$$\mathcal{H} ( x, y_1 , ... , y_p ) = ( h(x),
\tilde{H_1}(x, y_1 ), ..., \tilde{H_p}(x, y_p ))$$ \noindent where
$\tilde{H_i}$, $i=1, \dots, p$, are as in Definition
\ref{defcontactfamily}.

We call the bi-Lipschitz map germ $h$ of \emph{a common factor of
the $(p, q)$-multi pair of bi-Lipschitz homeomorphisms}.
\end{definition}

\begin{definition} Let  $q = ( q_1 , ... , q_p )$ be a multi index. Two $(p,q)$-multi pairs of map germs $F, G : ( \mathbb{R}^n ,0) \rightarrow
( \mathbb{R}^{q_1} \times ... \times \mathbb{R}^{q_p} , 0)$ are said
to be {\emph multi-$\mathcal{K}$-bi-Lipschitz equivalent} if there
exists a $(p,q)$-multi pair of bi-Lipschitz homeomorphism
$\mathcal{H} : ( \mathbb{R}^{n} \times \mathbb{R}^{q_1} \times ...
\times \mathbb{R}^{q_p}  ,0) \rightarrow \linebreak ( \mathbb{R}^{n}
\times \mathbb{R}^{q_1} \times ... \times \mathbb{R}^{q_p}  ,0)$ as
in Definition \ref{defH}, such that the following diagram is
commutative:
$$
\xymatrix{
( \mathbb{R}^{n},0) \ar[d]_h & \ar[r]^{({id}_{n} , F )} & & ( \mathbb{R}^n \times \mathbb{R}^{q_1} \times ... \times \mathbb{R}^{q_p} ,0 ) \ar[d]_{\mathcal{H}} & \ar[r]^{ {\pi}_{n}  } & & ( \mathbb{R}^{n},0) \ar[d]_h \\
( \mathbb{R}^{n},0) & \ar[r]_{({id}_{n} , G)} & & ( \mathbb{R}^n
\times \mathbb{R}^{q_1} \times ... \times \mathbb{R}^{q_p} ,0 ) &
\ar[r]_{{\pi}_{n}} & & ( \mathbb{R}^{n},0) }
$$  \\
where $id_n$ is the identity map germ of $\R^n$; $\pi_n$ is the
usual projection in $\R^n$; $h$ is the common factor of
$(p,q)$-multi pair of bi-Lipschitz homeomorphisms and for all $i =
1, ... ,p$ we have $\tilde{H}_i (\R^n \times {\{ 0 \} }^{q_i} ) =
\R^n \times {\{ 0 \} }^{q_i},$
\end{definition}

\begin{remark} Let $q = (q_1 , ...,q_p )$ be the multi index and  let $F = ( F_1 , ..., F_p ), G= ( G_1 , ... , G_p ):  ( \mathbb{R}^n ,0) \rightarrow
( \mathbb{R}^{q_1} \times ... \times \mathbb{R}^{q_p} , 0)$ be two
$(p,q)$-multi pairs of map germs. Then: \vspace{0,2cm}

i) When $q= q_1 \in \N$, the definition of
multi-$\mathcal{K}$-bi-Lipschitz equivalence coincides with the
definition of $\mathcal{K}$-bi-Lipschitz equivalence. \vspace{0,2cm}

ii) If  $F$ and $G$ are multi-$\mathcal{K}$-bi-Lipschitz equivalent
(as multi pair of map germs) then $F$ and $G$ are
$\mathcal{K}$-bi-Lipschitz equivalent (as map germs).\vspace{0,2cm}

2. If  $F$ and $G$ are multi-$\mathcal{K}$-bi-Lipschitz equivalent,
then the map germs $F_i$ and $G_i$ are $\mathcal{K}$-bi-Lipschitz
equivalent, for all $ i = 1 , ..., p.$
\end{remark}

\section{Main results and finiteness property}

\begin{theorem}\label{Teo123} Let $f = ( f_1 ,..., f_p ), g = ( g_1 ,..., g_p ) : ( \mathbb{R}^n ,0) \rightarrow ( \mathbb{R}^p ,0)$ be two
germs of Lipschitz maps.  Suppose  there exists a bi-Lipschitz
homeomorphism $H : ( \mathbb{R}^n \times \mathbb{R}^p , 0)
\rightarrow ( \mathbb{R}^n \times \mathbb{R}^p ,0)$ such
that:\vspace{0,15cm}

i) The sets \ $\mathbb{R}^n \times { \{ 0 \} }^{p} , \mathbb{R}^n
\times \mathbb{R}^{p-1} \times  \{ 0 \} , \mathbb{R}^n \times
\mathbb{R}^{p-2} \times  \{ 0 \} \times \mathbb{R} , ... ,
\mathbb{R}^n \times  \{ 0 \} \times  \mathbb{R}^{p-1}$ are invariant
under $H$. In other words, $H$ satisfies

$H(\mathbb{R}^n \times { \{ 0 \} }^{p}) = \mathbb{R}^n \times { \{ 0
\} }^{p}, H(\mathbb{R}^n \times \mathbb{R}^{p-1} \times  \{ 0 \}) =
\mathbb{R}^n \times \mathbb{R}^{p-1} \times  \{ 0 \}, \dots,
\linebreak  H(\mathbb{R}^n \times  \{ 0 \} \times
\mathbb{R}^{p-1})=\mathbb{R}^n \times  \{ 0 \} \times
\mathbb{R}^{p-1}.$

\smallskip

ii) $H ( graph (f) ) = graph(g)$. \vspace{0,15cm}

Then, the germs  $f$ and $g$ are
multi-$\mathcal{K}$-bi-Lipschitz-equivalent.

\end{theorem} \vspace{0,2cm}

To prove the Theorem \ref{Teo123} we need some preliminary results.

Let $f = ( f_1 ,..., f_p ), g = ( g_1 ,..., g_p ) : ( \mathbb{R}^n
,0) \rightarrow ( \mathbb{R}^p ,0)$ be two germs of Lipschitz maps
and suppose there exists a germ of bi-Lipschitz homeomorphism $H:
(\mathbb{R}^{n} \times \R^p,0) \rightarrow (\mathbb{R}^{n} \times
\R^p ,0)$, satisfying the conditions of Theorem \ref{Teo123}. Define
the following subsets of $\R^n \times \R^p$: \vspace{0,2cm}

$ V_{k}^{+} = \{ (x, y_1 , ... , y_p ) \in \mathbb{R}^n \times
\mathbb{R}^p \mid y_k > 0 \}$ \ \ and \vspace{0,2cm}

$ V_{k}^{-} = \{ (x, y_1 , ... , y_p ) \in \mathbb{R}^n \times
\mathbb{R}^p \mid y_k < 0 \}$, $k=1, \dots,p$.

\medskip

\textbf{Assertion 1.}\label{Af4} For each $k=1, \dots,p$, one of the
following conditions holds: \vspace{0,2cm}

i) $H( V_{k}^{+} ) = V_{k}^{+}$ \ \ and \ \ $H( V_{k}^{-} ) =
V_{k}^{-}$ or \vspace{0,2cm}

ii) $H( V_{k}^{+} ) = V_{k}^{-}$ \ \ and \ \ $H( V_{k}^{-} ) =
V_{k}^{+}$,

\begin{proof}If the condition does not hold, there are points $a, b, c, d \in \mathbb{R}^{n} \times \mathbb{R}^p$, such that $a \in  V_{k}^{+}$, $b, c, d \in  V_{k}^{-}$
with  $ H(a) = b, H(c) = d$. Consider a path $\lambda : [ 0,1]
\rightarrow V_{k}^{-}$ connecting the points $b$ and $d$. Therefore,
$H^{-1} \circ \lambda$ is a path in $\mathbb{R}^n \times
\mathbb{R}^{p}$, connecting the points $a$ and $c$  passing by
$\mathbb{R}^{n} \times \mathbb{R}^{k-1} \times \{ 0 \} \times
\mathbb{R}^{p-k}$. However, this is not possible because $H(
\mathbb{R}^{n} \times \mathbb{R}^{k-1} \times \{ 0 \} \times
\mathbb{R}^{p-k}) = \mathbb{R}^{n} \times \mathbb{R}^{k-1} \times \{
0 \} \times \mathbb{R}^{p-k}$ and $H(H^{-1} \circ \lambda)=\lambda \subset V_k^{-}$.
\end{proof}

\medskip

\textbf{Assertion 2.} Let $ h : (\mathbb{R}^{n},0) \rightarrow
(\mathbb{R}^{n} ,0)$ be defined by $ h(x) = {\pi}_{n} ( H(x, f (x) )
)$. Then $h$ is a bi-Lipschitz map germ.

\begin{proof}  Since $g$ is a Lipschitz map,
the projection $ {{\pi}_{n} }_{ {\mid}_{graf(g)} }$ is a
bi-Lipschitz map. By the same argument, a map $ x \mapsto ( x, f(x)
)$ is bi-Lipschiz. By hypothesis, the map $H$ is bi-Lipschitz. Hence
$h$ is bi-Lipschitz.
\end{proof}

\medskip

\textbf{Assertion 3.} One of the following assertions is true:

\medskip

i) $ f_i (x) \approx  \ g_i \circ h(x)$ or

ii) $ f_i (x) \approx - \ g_i \circ h(x)$, for all $i = 1, ..., p.$

\medskip

\begin{proof} Since $H$ is bi-Lipschitz, there exists two positives numbers  $c_1$ and $c_2$, such that \vspace{0,2cm}

$$ c_1 \vert f_1 (x) \vert  \leq  \parallel H(x, f_1 (x), ..., f_p (x)
) - H(x, 0 , f_2 (x), ..., f_p (x)  ) \parallel  \leq c_2 \vert f_1
(x) \vert.$$

By above construction,

$$\parallel H(x, f_1 (x), f_2 (x), ..., f_p (x) ) - H(x, 0 , f_2 (x), ... , f_p (x) ) \parallel =$$

$$= \parallel ( h(x), \ g_1 \circ h(x) \ , ...,  \ g_p \circ h(x) )
- H(x, 0 , f_2 (x), ... , f_p (x) ) \parallel \geq \ \vert g_1 \circ
h(x) \vert .$$

Therefore, $\vert g_1 \circ h(x) \vert \leq c_2 \vert f_1 (x)
\vert.$ Using the same procedure for the map $H^{-1}$, we obtain
that
$$ \vert g_1 \circ h(x) \vert \ \geq
\tilde{c}_1 \vert f_1 (x) \vert,$$ \noindent where $\tilde{c}_1$ is
a a real positive number.

Then, $ \tilde{c}_1 \vert f_1 (x) \mid  \ \leq \ \vert g_1 \circ
h(x) \vert \ \leq \ c_2 \vert f_1 (x) \vert .$

By Assertion \ref{Af4}  we have that for all $x \in \R^n$,
$$ sign ( {f}_{1}(x)) = sign ( {g}_{1} \circ h(x)) \,\,\,\,\, {\rm or} \,\,\,\,\,\, sign( {f}_{1}(x)) = - sinal( {g}_{1} \circ
h(x)).$$

Therefore, \ $f_1 (x) \approx \ g_1 \circ h(x) \  \mbox{ or } \  f_1
(x) \approx \ - g_1 \circ h(x)$.

Repeating the same process for all $i=1, \dots, p$, we obtain
$$f_i (x) \approx \ g_i \circ h(x) \ \mbox{ or }  \ f_i (x) \approx \ - g_i \circ h(x).$$
\end{proof} \vspace{0,2cm}

\begin{proof} [Proof of Theorem \ref{Teo123}] By Assertion 3 and Lemma \ref{lema1}, follows that  $f_i $ and  $g_i$ are $\mathcal{K}$-bi-Lipschitz equivalent,
for all $i = 1, ..., p.$ Then, for each $i=1, \dots, p$, there exist
a germ of bi-Lipschitz homeomorphisms $H_i : (\mathbb{R}^{n} \times
\mathbb{R},0) \rightarrow (\mathbb{R}^{n} \times \mathbb{R} ,0)$,
such that \vspace{0,2cm}

$i) \ H_i (x, y_i ) = ( h(x) , \tilde{H}_i (x, y_i ) ),$ with
$\tilde{H}_i:(\R^n \times \R,0) \to (\R,0)$, $i = 1, ..., p.$
\vspace{0,2cm}

$ii) \ H_i (x, f_i(x) ) = ( h(x) , g_i \circ h(x) )$, $i = 1, ...,
p.$ \vspace{0,2cm}

$iii) \ H_i ( \mathbb{R}^{n} \times \{ 0 \}  ) = \mathbb{R}^{n}
\times \{ 0 \} $, $i = 1, ..., p.$ \vspace{0,2cm}

Define the map $ \mathcal{H} : (\mathbb{R}^{n} \times \R^p,0)
\rightarrow (\mathbb{R}^{n} \times \R^p ,0)$ given by \vspace{0,2cm}

$$ \mathcal{H}(x, y_1 , ... , y_p ) = ( h(x) , \tilde{H}_1 (x, y_1 ), ... , \tilde{H}_p (x, y_p
)).$$

Considering $q=(1,\dots,1)$ $p$-times, we obtain that
$\mathcal{H}$ is a $(p, q)$-multi pair of bi-Lipschitz homomorphism,
generated by the contact family of bi-Lipschitz homeomorphisms, $\{
h, H_1 , ..., H_p \}$. Moreover, \vspace{0,2cm}

$i) \ \mathcal{H}(x, f_1 (x) , ... , f_p (x) ) = ( h(x) , g_1 \circ
h(x), ... , g_p \circ h(x) )$ and \vspace{0,2cm}

$ii) \ \mathcal{H} ( \mathbb{R}^{n} \times { \{ 0 \}}^{p}  ) =
\mathbb{R}^{n} \times { \{ 0 \}}^{p} $.\vspace{0,2cm}

Hence, $f$  and  $g$ are multi-$\mathcal{K}$-bi-Lipschitz
equivalent.
\end{proof}

\begin{theorem}\label{mainthm}(Finiteness theorem) Let $P^{k}(n,p) $ be the set of all real
polynomial map germs $ f = ( f_1 ,..., f_p ) :(\mathbb{R}^{n},0)
\rightarrow (\mathbb{R}^{p},0), $ with degree of $ f_1 ,..., f_p$
less than or equal to $k\in\N$. Then the set of the equivalence
classes of $P^{k}(n,p) $, with respect to
multi-$\mathcal{K}$-bi-Lipschitz-equivalence is finite.
\end{theorem}

\begin{proof} Let $F:(\R^n,0) \to (\R^p,0) \in  P^{k}(n,p) $. We associate to $F$ the following family of algebraic subsets:
$$
X_{F} = \{ \ \mathbb{R}^n \times {\{ 0 \}}^p  \ , \  \mathbb{R}^n
\times \mathbb{R}^{p-1} \times \{ 0 \}  \ , \  \mathbb{R}^n \times
\mathbb{R}^{p-2} \times \{ 0 \} \times \mathbb{R} \ ,  \ldots  , \
\mathbb{R}^n \times \{ 0 \} \times \mathbb{R}^{p-1} \ , \ graf(F)
\}.
$$

Consider the set of all the families of algebraic subsets defined as
above, that is,
$$
\mathcal{F} = \{X_{F_1} \ , \ X_{F_2}, \dots,  \ X_{F_i}, \dots\}.
$$

We define the following equivalence relation in $\mathcal{F}$:
\medskip

$X_{F_i}$ and $X_{F_j}$  are called multi-$\mathcal{V}$-bi-Lipschitz
equivalent if there exists  a bi-Lipschitz homeomorphism $H : (
\mathbb{R}^n \times \mathbb{R}^p , 0) \rightarrow ( \mathbb{R}^n
\times \mathbb{R}^p ,0)$ satisfying the conditions i) and ii) of
Theorem \ref{Teo123}.

By Vallete  Lipschitz Triviality Theorem (cf. \cite{V}), the number
of equivalence classes  with respect to
multi-$\mathcal{V}$-bi-Lipschitz equivalence, is finite. Applying
the Theorem \ref{Teo123} follows that  $F_i$ and $F_j$ are
multi-$\mathcal{K}$-bi-Lipschitz equivalent and so we conclude the
proof of theorem.
\end{proof}

\begin{corollary} The set of the equivalence classes of
$P^{k}(n,p)$, with respect to $\mathcal{K}$-bi-Lipschitz
equivalence, is finite.
\end{corollary}

\begin{proof}
Let $F,G \in P^{k}(n,p)$. By Remark \ref{remark}, $F$ and $G$ can be
considered a $(p,q)$-multi pair of function germs (through its
coordinate functions). Since multi-$\mathcal{K}$-bi-Lipschitz
equivalence admits the finiteness property, we can suppose that the
corresponding families of coordinate functions are
multi-$\mathcal{K}$-bi-Lipschitz equivalent. Then one can consider a
map $H : ( \mathbb{R}^n \times \mathbb{R}^p , 0) \rightarrow (
\mathbb{R}^n \times \mathbb{R}^p ,0)$ defined as follows:
$$
H=(h,\tilde{H}),
$$
\noindent where $h$ and $\tilde{H}=(\tilde{H}_1, \dots,
\tilde{H}_p)$ such that $h$ and $\tilde{H}_i$ are the maps obtained
from the corresponding multi-$\mathcal{K}$-bi-Lipschitz equivalence
of the coordinate functions. Then $h$ and $H$ are bi-Lipschitz
homeomorphisms and satisfy the conditions of Definition
\ref{defKbil} for $F$ and $G$. Hence, $F$ and $G$ are
$\mathcal{K}$-bi-Lipschitz equivalent.

Since in the proof we use the finiteness property of the
multi-$\mathcal{K}$-bi-Lipschitz equivalence we have that the number
of $\mathcal{K}$-bi-Lipschitz classes is also finite.
\end{proof}

\end{document}